\documentclass{amsart}

\newtheorem{thm}{Theorem}[section]
\newtheorem{cor}[thm]{Corollary}
\newtheorem{lem}[thm]{Lemma}
\newtheorem{prop}[thm]{Proposition}
\theoremstyle{definition}
\newtheorem{defn}[thm]{Definition}
\theoremstyle{remark}

\newtheorem{exm}[thm]{Example}
\numberwithin{equation}{section}

\newcommand{\irr}{{\rm Irr}}

\newcommand{\cl}{{\rm Cl}}
\newcommand{\bl}{{\rm Bl}}

\begin{document}

\title[Block Form of Frobenius Groups]
 {Block Form of Frobenius Groups}

\author{Jiwen Zeng*}
\address{School of Mathematic Science, Xiamen University, Xiamen, Fujian, 361005, China}
\author{Jiping Zhang}
\address{School of Mathematic Science, Peking University, Beijing, 100871, China}
\email{jwzeng@xmu.edu.cn}
\thanks{* Corresponding author: Jiwen Zeng, jwzeng@xmu.edu.cn}



\date{}

\dedicatory{}


\begin{abstract}
  The aim of this paper is  to apply character properties of Frobenius group to a local block form of an group algebra. We start by establishing a block form of Brauer permutation Lemma by using block participation of conjugate classes of a group $G$. Then we can define a pair of Frobenius corresponding blocks between a group $G$ and its normal subgroup $N$. A near group condition is given to determine a pair of Frobenius corresponding blocks. With a pair of Frobenius corresponding blocks, we study its group structure. At last we prove connections between nilpotent properties and Frobenius corresponding blocks.
\end{abstract}

\maketitle

\section*{Introduction}
Frobenius groups are well-known in the area of finite group study. There are at least five equivalent conditions to define Frobenius groups. Hence we can find a lot of generalization of Frobenius groups. If you generalize Frobenius group in one way, you could not find the other ways to define the same generalized Frobenius group. This is one of interests to study Frobenius groups. Our interest is the generalization of Frobenius groups in character theory. Related paper refer to \cite{ekr}\cite{fdz}\cite{cz}\cite{lew}.

Suppose $N$ is a normal subgroup of a group $G$. Let Irr$(N)$ denote the set of all irreducible ordinary characters of $N$. If any $\varphi \in \textrm{Irr}(N)$ is induced irreducibly to $\varphi^{G} \in \textrm{Irr}(G)$, we say $G$ is Frobenius group and $N$ is the kernel. When you consider an block algebra $b$ of $N$, naturally we would like to consider local character property: is any $\varphi \in \textrm{Irr}(b)$ induced irreducibly to $G$?

As usual, a property  held for a block algebra is said to be a local property for group algebras. If any character $\varphi$ from a block algebra $b$ of normal subgroup $N$ of $G$ is induced  irreducibly to $G$, we say that the property is local Frobenius property. According to the properties of Brauer correspondence, the induced irreducible character $\varphi^{G}$ belongs to a block $B=b^{G}$, which is called the Brauer correspondence of $b$. Or we say $b^{G}$ is defined. Hence it is easily to define  $(b, b^{G})$ as a pair of Frobenius corresponding blocks.

If $G$ is a Frobenius group with kernel $N$, any pair $(b, b^{G})$ will be  a pair of Frobenius corresponding blocks for a block $b$ of $N$. If there exists a pair of Frobenius corresponding blocks $(b,b^{G})$ for a group $G$ and its normal subgroup $N$, how much can we say the group $G$ and its normal subgroup $N$? In this paper, we are going to study the problem in  three respects.

First, we extend Brauer permutation Lemma to a local block form by using block participation of conjugate classes\cite{br}. Then we can give a near group condition to decide a pair of Frobenius corresponding blocks. Second, we will study the structure of the group if it has a pair of Frobenius corresponding blocks. Some results  of lower defect groups\cite{ols} and Brauer correspondence \cite{bra}\cite{broue2} will be used for the purpose of studying these groups with a pair of Frobenius corresponding blocks.  At last, we also obtain some results about nilpotent properties of Frobenius corresponding blocks. These results partly answer some questions in paper\cite{bz1}\cite{bz2}.

In section 5, We have examples to show that there exists non-Frobenius finite group with a pair of Frobenius corresponding blocks. Hence we have a good reason to define Frobenius corresponding blocks and study them.

In this paper, most notations and basic definitions can refer to ~\cite{naga}.
Let $G$ be a finite group. Let Irr($G$) and IBr($G$) denote the set of irreducible ordinary characters and irreducible Brauer characters of $G$, respectively.

The following notation and terminology will be used throughout in this paper. $G$ is a finite group. $p>0$ is a prime number. $(K, R, F)$ is a splitting $p$-module system, that means $R$ is a complete discrete valuation ring with a maximal ideal $\pi$ such that $F=R/\pi$ is a field of characteristic $p$ and $K$ is the quotient  field of $R$. A $p$-regular element means an element of $G$ whose order
is prime to $p$. $G^{0}$ is the set of all $p$-regular elements in $G$. Bl($G$) denotes the set of $p$-blocks of $G$. $B_{0}$ denotes the principal block of $G$, containing the trivial character. We use LBr($G$) to denote the set of all linear Brauer characters of $G$.

Our paper is mostly self-contained. Although some basic results are well-known and could be found in literatures, we like to bring  convenience to  readers for the proof. Our methods of proofs in this paper most come from modular representation of group theory.

We have six sections. Basic results and notations appear in the first two sections and they could be found in different literatures.  Section 1 is Brauer permutation Lemma and its proof. Section 2 is block partition of conjugate classes of group $G$. In section 3 we will give block form of Brauer permutation Lemma which localizes the global Brauer permutation Lemma for our use in the following sections. Lemma 3.1 is our main result in Section 3. In section 4, we will define a pair of Frobenius corresponding blocks and our main result is Theorem 4.6, which proves that near group condition can decide a pair of Frobenius corresponding blocks. Section 5 will give our main results: Theorem 5.4 and 5.5 in this paper. Here we discuss the structure of finite groups with a pair of Frobenius corresponding blocks.  In the last section, we study the connection between nilpotent properties and Frobenius corresponding blocks.
\smallskip

\section{Brauer Permutation Lemma}

In this section we state the well-known Brauer permutation Lemma and its application to character theory. We start with a few definitions and introduce notation and terminology that is consistent throughout this paper.

Given a set $\Omega=\{1,2,...,n\}, S_{\Omega}$ is the permutation group of $\Omega.$
Let a group $G$ act on the set $\Omega=\{1,2,...,n\}$ by $\rho: x\rightarrow \rho(x) \in S_{\Omega}$. we denote by:
$$\rho(x)=\left(
  \begin{array}{cccc}
    1 & 2 & \cdots & n \\
    1^{x} & 2^{x} & \cdots & n^{x} \\
  \end{array}
\right)$$
This defines a permutation matrix by:
$$a^{x}_{ij}=\left\{\begin{array}{cc}
                 1 & j=i^{x}\\
                 0 & {\rm otherwise}\\
              \end{array}\right\}$$

  The following Brauer permutation Lemma should appear in many literatures. Here for the reader's convenience, we give a proof by using modular representation of group theory.
\begin{lem}(Brauer Permutation Lemma)
Suppose $\Omega=\{1,2,...,n\},$ and $G$ acts on $\Omega$ by $\rho$ and $\sigma$, respectively. If there exists a matrix
 $A=(a_{ij})\in {\rm GL}(K)$ such that $\rho(x)$ permutes the rows of $A$ and $\sigma(x)$ permutes the columns of $A$ satisfying
 $$a_{i^{\rho(x)}j^{\sigma(x)}}=a_{ij}$$
 for all $x,i,j$, then the following holds:
 \begin{itemize}
   \item for any $x \in G,$ the number of rows fixed by  $\rho(x)$ is equal to that of columns fixed by $\sigma(x)$
   \item the number of  $G$-orbits on the set of the rows of $A$ equals to that of $G$-orbits on the set of the columns of $A$
 \end{itemize}
\end{lem}

\begin{proof}
Consider $K\Omega$ as $G-$permutation module by actions $\rho(x)$ and $\sigma(x)$, respectively. It suffices to prove they are isomorphic as $G-$module. Take a map
$$f:   K\Omega \longrightarrow K\Omega           $$
by matrix $A$. Note $\Omega$ is a basis of $K\Omega$. Then $f$ is a one to one map. To check $G$-homomorphism, we need to prove
$$f(i^{\rho(x)})=f(i)^{\sigma(x)}$$
Using the corresponding matrix to express:
$$A\sigma(x)=\rho(x)A$$
This is equivalent to  the following:
$$\begin{array}{c}
   \rho(x^{-1}) A\sigma(x)=A \\
   \rho(x^{-1})=(\rho^{x^{-1}}_{ij}),\sigma(x)=(\sigma^{x}_{ij}) \\
  \end{array}
$$
Note that
$$\rho^{x^{-1}}_{ij}=1\Leftrightarrow j=i^{\rho(x^{-1})}\Leftrightarrow i=j^{\rho(x)},\rho^{x^{-1}}_{j^{\rho(x)}j}=1$$
It means $\rho(x^{-1})=(\rho^{x^{-1}}_{ij})$ acting on the rows on the identity matrix by:
$$i\rightarrow i^{\rho(x)},i=1,2,...,n$$
Hence $\rho(x^{-1})A=(a_{i^{\rho(x)}j})$.
Similarly, note
$$\sigma^{x}_{ij}=1\Leftrightarrow j=i^{\sigma(x)},\sigma^{x}_{ii^{\sigma(x)}}=1$$
It means $\sigma(x)=(\sigma^{x}_{ij})$ acting on the columns on the identity matrix by:
$$j\rightarrow j^{\sigma(x)},j=1,2,...,n$$
Hence
 $\rho(x^{-1})A\sigma(x)=(a_{i^{\rho(x)}j^{\sigma(x)}})$ and we have $\rho(x^{-1})A\sigma(x)=A$
\end{proof}

Let us denote the set of conjugate classes of $G$ by ${\rm Cl}(G$). If a group $A$ acts on the group $G$ by automorphism, the group $A$ naturally acts on the set ${\rm Irr}(G)$  by $\chi^{a}(x)=\chi(x^{a^{-1}}).$
The following result is the well-known application of Brauer permutation Lemma to character theory.

\begin{cor}Let a group $A$ acts the set $\irr(G)=\{\chi_{1},\chi_{2},...,\chi_{n}\} $ and the set ${\rm Cl}(G)=C_{1},C_{2},...,C_{n}$ satisfying:
$$\chi_{i}^{a}(x_{j}^{a})=\chi_{i}(x_{j}),x_{j}\in C_{j},a\in A,i,j=1,2,...,n,x_{j}^{a}\in C_{j}^{a}$$
Then the following holds:
\begin{itemize}
  \item for any $a\in A$, the number of fixed irreducible characters in ${\rm Irr}(G)$ equals that of fixed classes in ${\rm Cl}(G)$.
  \item the number of $A$-orbits in ${\rm Irr}(G)$ equals that of $A$-orbits in ${\rm Cl}(G)$.
\end{itemize}
\end{cor}
\begin{proof}
Note matrix $X=(\chi_{i}(x_{j}))$ is nonsingular. Consider the situations when $A$ acts on the rows and columns of the matrix $X$. Then the conditions of Brauer permutation Lemma are satisfied.
\end{proof}

\section{ Block Partition of conjugate classes of group $G$}
In this section we state Brauer's results\cite{br} of block partition of conjugate classes of a group $G$. Some related results refer to \cite{ii,broue,ols}.
Although their proof can be found in literature, we give a proof for the reader's convenience.
\begin{thm} (Brauer)Let $V$ be a free $R$-module with a basis $\Omega=\{v_{1},v_{2},...,v_{n}\}$. Given a direct sum decomposition:
\begin{equation}\label{ol}
  V=W_{1}\bigoplus W_{2}\bigoplus \cdots\bigoplus W_{r}
\end{equation}
and a Projection $\Pi_{i}$ of $V$ on $W_{i},i=1,2,...,r$, there exists a partition:
\begin{equation}\label{ols}
  \Omega=\bigcup_{i=1,...,r}\Omega_{i}, \Omega_{i}\bigcap \Omega_{j}=\emptyset,i\neq j.
\end{equation}
such that $\Pi_{i}(\Omega_{i})$ is a basis of $W_{i},i=1,2,...,r.$
\end{thm}
\begin{proof}
Let $\{w^{j}_{1},w^{j}_{2},...,w^{j}_{n_{j}}\}$ be a basis of $W_{j},j=1,2,...,r.$ Write
\begin{equation*}
  v_{i}=\sum_{j,k}\alpha_{i,(j,k)}w_{k}^{j},i=1,2,...,n.
\end{equation*}
Then $$\Pi_{j}(v_{i})=\sum_{k=1,...,n_{j}}\alpha_{i,(j,k)}w_{k}^{j}
$$
Let $A=(\alpha_{i,(j,k)})$ be the $n\times n$ matrix with $i$ as the row index and $(j,k)$ as the columns in the order of basis of $W_{1},W_{2},...,W_{r}.$

Since the first $n\times n_{1}$ sub-matrix of $A$ has a rank $n_{1}$, we can rearrange the rows of $A$ to get a new matrix $A_{1}$, which has first principal  $n_{1}\times n_{1}$ minor as nonsingular. Then consider lower right corner $(n-n_{1})\times (n-n_{1})$ sub-matrix $B$ of $A_{1}$. Proceeding  as above, we can get second principal $n_{2}\times n_{2}$  minor as nonsingular. Continue the process above, we get a matrix which has principal minors:
$$n_{1}\times n_{1},n_{2}\times n_{2},...,n_{r}\times n_{r} $$
all of which are nonsingular. Corresponding the new matrix, we have a partition of
$$\Omega=\bigcup_{i=1,2,...,r}\Omega_{i}, \Omega_{i}\bigcap\Omega_{j}=\emptyset,i\neq j$$
with $\Pi_{j}(\Omega_{j})$ is a basis of $W_{j},j=1,2,...,r.$
\end{proof}

The partition of $\Omega$ as in the Theorem above is said to be a partition associated with the decomposition of $V=\bigoplus_{i=1,...,r} W_{i}.$

\begin{cor}
With the same notations and conditions above, given a partition as Theorem 2.1, if $U$ is a $R$-submodule of $V$ with a basis $\Lambda$, satisfying:
\begin{itemize}
  \item $\Pi_{j}(U)\subseteq U,j=1,2,...,r$
  \item $\Lambda \subseteq \Omega$
\end{itemize}
Then $\Lambda=\bigcup_{i=1,...,r}(\Lambda\bigcap\Omega_{i})$ is a partition associated with the decomposition $U=\bigoplus_{i}\Pi_{i}(U)$. In particular,
${\rm Rank}_{R}\Pi_{i}(U)=|\Lambda\bigcap\Omega_{i}|$
\end{cor}

The following results and notations is standard as usual(refer to \cite[Chapter 5]{naga}).
\begin{description}
\item[Block decomposition] let
$$V=Z(RG)=\bigoplus_{B\in \bl(G)}e_{B}Z(RG),W_{B}=e_{B}Z(RG).$$
 The projective map from $Z(RG)$ is:$Z(RG)\rightarrow e_{B}Z(RG)$
  \item[Basis of $Z(RG)$] the set of conjugate class sums of $$\Omega=\{\widehat{C}|C\in {\cl}(G),\widehat{C}=\sum_{x\in C}x\} $$
  \item[Rank of $e_{B}Z(RG)$]Note:$K$ is the factor field of $R$. In the $KG$, we have $e_{B}=\sum_{\chi\in {\irr}(B)}e_{\chi}$ and Rank$_{R}(e_{B}Z(RG))=$Dim$_{K}(e_{B}Z(KG))=k(B)$.
\end{description}

\begin{thm}\cite[Chapter 5]{naga}There is a partition of ${\cl}(G)$, the set of conjugate classes of $G$,
\begin{equation}\label{oles}
  {\cl}(G)=\bigcup_{B\in \bl(G)}\Omega_{B}, \Omega_{B}\bigcap\Omega_{B'}=\emptyset,B\neq B'
\end{equation}
such that the following holds:\begin{itemize}
                                \item $\{e_{B}\widehat{C}|C\in \Omega_{B}\}$ is a $R$- basis of $e_{B}Z(RG)$, hence $k(B)=|\Omega_{B}|$
                                \item the value of $\nu$ at det$(\omega_{\chi}(\widehat{C}))_{\chi\in B,C\in \Omega_{B} }$ is not large than any other values of
                                $\nu$ at the $k(B)\times k(B)$ minors of the $k(B)\times k$ matrix det$(\omega_{\chi}(\widehat{C}))_{\chi\in B,C\in {\cl}(G) },k=\cl(G)$.
                              \end{itemize}

\end{thm}
\begin{proof}
 The first case is a clear result from  Theorem 2.1.

 For any $C\in \cl(G)$, we have
 $$e_{B}\widehat{C}=\sum_{C_{i}\in \Omega_{B}}a_{CC_{i} }e_{B}\widehat{C_{i}},a_{ CC_{i}}\in R$$
 Hence  for $\chi\in \irr(B),C\in \cl(G)$, there is
$$\omega_{\chi}(\widehat{C})=\omega_{\chi}(e_{B}\widehat{C})=\sum_{C_{i}\in \Omega_{B}}a_{CC_{i}}\omega_{\chi}(e_{B}\widehat{C_{i}}),a_{CC_{i}}\in R$$
This is a matrix relation, (just taking $k(B)$ conjugate classes from Cl$(G)$):
$$              (\omega_{\chi}(e_{B}\widehat{C}))=(\omega_{\chi}(e_{B}\widehat{C_{i}}))(a_{CC_{i}})^{T}$$
by which we have the second assertion, as $\omega_{\chi}(e_{B}\widehat{C})=\omega_{\chi}(\widehat{C}).$
\end{proof}

\begin{description}
  \item[Block partition of $\cl(G)$] The partition of $\cl(G)$ in the Theorem 2.3 is called Block partition of  $\cl(G)$.
\end{description}

\bigskip
\section{ Block form of Brauer Permutation Lemma}

In this section, we are going to localize the globe Brauer permutation Lemma for our use in the following sections. The previous two sections allows us to attain our purpose.

 Suppose that a group $A$ acts on the group $G$. A block partition of Cl$(G)$ is given as in Theorem 2.3.
Naturally the action of an element $a$ of $A$ on $G$ produces an automorphism
$$RG: \sum_{g\in G}a_{g}g\mapsto  \sum_{g\in G}a_{g}g^{a}$$
It is also clear that $Z(RG)=Z(RG)^{a}$, that is to say, the center of $Z(RG)$ is fixed under automorphism.

An action of an element $a$ of $A$ produces a permutation on the set Bl$(G)$ and block partition $\{\Omega_{B}|B\in G\}$. This is because
$$
\begin{array}{c}
   e_{B}\mapsto e_{B}^{a} \\
  \Omega_{B}\mapsto \Omega_{B}^{a}=\{C^{a}|C\in \Omega_{B}\}
\end{array}
$$
Since $e_{B}\Omega_{B}$ is a basis of $e_{B}Z(RG),e_{B}^{a}\Omega_{B}^{a}$ is a basis of $e_{B}^{a}Z(RG)$. Hence we say:
$$  \cl(G)=\bigcup_{B}\Omega_{B}^{a}$$
is also a block partition of Cl$(G)$.

Suppose $\{e_{B}^{a}|a\in A\}$ is the orbit of $e_{B}$ under action of $A$ on $G$, then $\{\Omega_{B}^{a}|a\in A\}$ is also an orbit of $A$-action.
Let $A_{B}$ denote the fixed subgroup of $e_{B}$ in $A$ and $A/A_{B}$ is the coset of $A_{B}$ in $A$. Then we have the following clear facts:
\begin{description}
  \item[Center idempotent] we denote $f_{B}=\sum_{a\in A/A_{B}}e_{B}^{a}$, which is an center idempotent of $RG.$
  \item[Set of Irr$f_{B}$] Irr$(f_{B})=\bigcup_{a\in A/A_{B}}{\rm Irr}(B^{a})$.
  \item[Set of conjugate classes corresponding to Irr$f_{B}$] $\Omega_{f_{B}}=\bigcup_{a\in A/A_{B}}\Omega_{B}^{a}.$
\end{description}
By the block partition of Cl$(G),$ we know that
$$f_{B}Z(RG)=\bigoplus_{a\in A/A_{B}}e_{B}^{a}Z(RG)$$
 and it has a basis
$$\bigcup_{a\in A/A_{B}}e_{B}^{a}\Omega_{B}^{a}=\{e_{B}^{a}\widehat{C^{a}}|a\in A/A_{B},C^{a}\in \Omega_{B}^{a}\}
$$
Suppose $\chi \in {\rm Irr}(B)$. We have
$$\chi^{a}(x^{a})=\chi^{a}(e_{B}^{a}x^{a})=\chi(e_{B}x)=\chi(x),x\in C$$
since $\chi^{a}\in {\rm Irr}(B^{a}).$
\begin{description}
  \item[$k(B)$] denote the number of characters in Irr$(B)$ for a block $B$ in $G$.
\end{description}

In a word, we have a summary of  all arguments above:
 \begin{thm}\label{3.1}With the same notations and conditions above, suppose that a group $A$ acts on the group $G$. For a block $B \in {\bl }(G)$,  $A$ naturally acts on the set ${\rm Irr }(f_{B})$ and the set $\Omega_{f_{B}}$
 satisfying:
$$\chi_{i}^{a}(x_{j}^{a})=\chi_{i}(x_{j}),x_{j}\in C_{j},a\in A,i,j=1,2,...,m,x_{j}^{a}\in C_{j}^{a}$$
where $m=\sum_{a\in A/A_{B}}k(B^{a})=|A/A_{B}|k(B)$.
Then the following holds:
\begin{itemize}
  \item for any $a\in A$, the number of fixed irreducible characters in ${\irr}(f_{B})$ equals that of fixed classes in $\Omega_{f_{B}}$.
  \item the number of $A$-orbits in ${\irr}(f_{B})$ equals that of $A$-orbits in $\Omega_{f_{B}}$.
\end{itemize}
 \end{thm}

 In particular, if $B$ is fixed under $A$-action, we have
\begin{cor}
With the same notations and conditions above, suppose that a group $A$ acts on the group $G$. For a block $B \in {\bl }(G)$,  if $B$ is fixed under $A$-action, then $A$ naturally acts on the set \irr ($B$) and the set $\Omega_{B}$
 satisfying:
$$\chi_{i}^{a}(x_{j}^{a})=\chi_{i}(x_{j}),x_{j}\in C_{j},a\in A,i,j=1,2,...,m,x_{j}^{a}\in C_{j}^{a}$$
where $m=k(B)$,
and the following holds:
\begin{itemize}
  \item for any $a\in A$, the number of fixed irreducible characters in $\irr(B)$ equals that of fixed classes in $\Omega_{B}$.
  \item the number of $A$-orbits in $\irr(B)$ equals that of $A$-orbits in $\Omega_{B}$.
\end{itemize}
\end{cor}

\bigskip

 \section{ Frobenius Corresponding Block}

From this section in whole paper, we suppose $N$ is a normal subgroup of $G$. Bl$(N)$ is the set of blocks of $N$. If $b$ is a block of $N$, we use $e_{b}$ to denote the block idempotent of $b$.
Consider the group $G$ acts on its normal subgroup $N$ by conjugation. Suppose we have block partition of Cl$(N)$, denoted by
$$\Omega=\bigcup_{b\in {\rm Bl}(N)}\Omega_{b}$$
We keep to use the same notations  as before. For a block $b$ of $N$, let
$$T(b)=\{x\in G|b^{x}=b\}$$
 denote the inertial group of $b$. Then
\begin{description}
\item[Center idempotent ]$f_{b}=\sum_{x\in G/T(b)}e_{b}^{x}, e_{b}^{x}=e_{b^{x}}$
  \item[Set of irreducible characters of $f_{b}$] Irr$(f_{b})=\bigcup_{x\in G/T(b)}{\rm Irr}(b^{x})$
  \item[Set of classes of $f_{b}$ by block partition]$\Omega_{f_{b}}= \bigcup_{x\in G/T(b)}\Omega_{b}^{x}=\Omega_{b^{x}}$
\end{description}
Note $f_{b}$ is also a center idempotent of $RG$ and it has a block idempotent decomposition in $Z(RG)$:
$$f_{b}=\sum_{B}e_{B}$$
As usually the block $b$ is said to be covered by a block $B$ if $e_{B}$ appears in the sums above.

Now we give a definition of Frobenius corresponding block.
\begin{defn}
With the same notations as above, suppose $N$ is a normal subgroup of $G$. For $b \in {\rm Bl}(N)$ and $B \in {\rm Bl}(G)$,
we use $(b,B)$ denote $b$ covered by $B$. For any nontrivial $\chi \in {\rm Irr}(b)$, if  induced $\chi^{G}$ as $G$ character is irreducible, we say
$(b,B)$ is a pair of Frobenius corresponding blocks, or $b$ is Frobenius block corresponding to $B$.
\end{defn}

Note: Since $b$ is isomorphic to $b^{x},x\in G$ by $G$-conjugation, $(b,B)$ is a pair of Frobenius corresponding blocks if and only if $(b^{x},B)$ is a pair of
Frobenius corresponding blocks.

For example, If $G$ is a Frobenius group with kernel $N$, then all $(b,B)$ is a pair of Frobenius corresponding blocks.

What is the relation between Irr$(b)$ and Irr$(B)$ when $b$ is covered by $B$? The following result should be well-known(Nagao's book, Page P337-339, Lemma 5.4, Lemma 5.7, Lemma 5.8):
\begin{lem} Let $B$ covers $b$. Then
\begin{itemize}
\item $B$ covers $b$, if and only if: any $\chi \in {\rm Irr}(B) $, all irreducible constituent of $\chi_{N}$ belong to some Irr$(b^{x})$.
  \item Let $\chi\in {\rm Irr}(B)$. Then for each $x\in G/T(b)$,  there exists $\varphi \in {\rm Irr}(b^{x})$ as an irreducible constituent of $\chi_{N}$
  \item If $\varphi \in {\rm Irr}(b)$, there is always $\chi\in {\rm Irr}(B)$ such that $\varphi$ is an irreducible constituent of $\chi_{N}$
  \item if  $\varphi \in {\rm Irr}(b)$, there exists $\chi\in {\rm Irr}(B)$ as an irreducible constituent of $\varphi^{G}$
  \item $B$ covers $b$, if and only if any  $\varphi \in {\rm Irr}(b)$, there exists $\chi \in {\rm Irr}(B) $ as an irreducible constituent of $\varphi^{G}$
\end{itemize}
\end{lem}

Consider block cover is correspondence from  a set of conjugate blocks in $N$ to more blocks in $G$.
\begin{cor}
If $(b,B)$ is a pair of Frobenius corresponding blocks, then $B$ is only one block to cover $b$.
\end{cor}
\begin{proof}
Since $\varphi \in {\rm Irr}(b)$ is induced to an irreducible character $\varphi^{G}$, it belongs to a unique block $B$ of $G$.
\end{proof}

Note that $f_{b}$ is a primitive idempotent of $(RN)^{G}$. Hence we also have
\begin{cor}
If $(b,B)$ is a pair of Frobenius corresponding blocks, then the block idempotent $e_{B}$ of $B$ belongs to $(RN)^{G}$.
\end{cor}
\begin{proof}
This is because
$$e_{B}=f_{b}=\sum_{x\in G/T(b)}e_{b}^{x}$$
\end{proof}

By Theorem 3.1, the block form of Brauer permutation Lemma, we have a equivalent condition for a pair of Frobenius corresponding blocks.
\begin{description}
  \item[Inertial group $T(\chi)=$]$\{x\in G|\chi^{x}=\chi\} , \chi \in {\rm Irr}(b)$
\end{description}

\begin{thm}
With the same notations as above, $(b,B)$ is a pair of Frobenius corresponding blocks if and only if any $1\neq C \in \Omega_{b},C$ is only fixed by $N$ under $G$-conjugate action.
\end{thm}
\begin{proof}
Note that for any $1\neq\chi\in {\rm Irr}(b), \chi^{G}$ is irreducible if and only if its inertial group $T(\chi)$ is $N$. Hence by Theorem 3.1, we have the assertion.
\end{proof}

By using block partition of the set Cl$(N)$ of conjugate classes of $N$, We give a pure group condition for a pair of Frobenius corresponding blocks.
\begin{thm}
Suppose $N$ is normal subgroup of $G$ and $b$ is a block of $N$. Let $\Omega_{b}$ be the set of these conjugate classes decided by block partition of ${\rm Cl}(N)$. Then $(b,B)$ is a pair of Frobenius corresponding blocks if and only if: for any $1\neq a\in C, C\in \Omega_{b}, C_{G}(a)\leq N.$
\end{thm}
\begin{proof}
By Theorem 4.5, it is suffice to prove the condition is equivalent to: for any $1\neq C\in \Omega_{b}, C$ is only fixed by $N$ under $G$-conjugate action. We denote the fixed-point group $C$ by $C_{G}(C).$ First we claim that:
$$        C_{G}(a)\leq C_{G}(C), a\in C$$
Let $x\in C_{G}(a)$, then$$ \begin{array}{c}
                            a^{x}=a \\
                            \Rightarrow \forall n\in N, a^{nx} =a^{xn^{x}}=a^{n^{x}},n^{x}\in N\\
                            \Rightarrow a^{nx}\in C\\
                            \Rightarrow x\in C_{G}(C)
                          \end{array}
$$
Our claim is proved. Hence
$$C_{G}(C)\leq N\Rightarrow C_{G}(a)\leq N, \forall a\in C.$$

Conversely, suppose $C_{G}(a)\leq N, a\in C$. We want to prove $C_{G}(C)\leq N.$ By using contradictions, suppose $x\in C_{G}(C)-N.$
Since $a^{x}=a^{n}\in C,n\in N$, we have
$$\begin{array}{c}
    a^{xn^{-1}}=a \\
    \Rightarrow  xn^{-1}\in N \\
    \Rightarrow x=xn^{-1}n\in N
  \end{array}
$$
which is a contradiction. Hence $C_{G}(C)\leq N.$
\end{proof}

\section{ Some properties of Frobenius corresponding blocks}

Keep to use the same notations as before. Let $N$ be a normal subgroup of $G$. We denote the principal block of $G$ by $B_{0}$: containing trivial character $1_{R}.$ Similarly, $b_{0}$ is the principal block of $N$.

\begin{lem}Suppose $(b_{0},B_{0})$ is a pair of Frobenius corresponding blocks. Then ${\rm Irr}(G/N)$ belongs to principal $B_{0}$.
\end{lem}
\begin{proof}
Take any $\chi \in {\rm Irr}(G/N)$ and suppose $\chi \in  {\rm Irr}(B'),B' \in {\rm Bl}(G).$  Since $\chi_{N}$ have an irreducible constituent in common with $1_{R}$, then $B'$ and $B_{0}$ cover the same block $b_{0}$(Refer to Nagao's book, Page 340, Theorem 5.9). By Corollary 4.3 $B_{0}$ is the unique block to cover $b_{0}$. Hence $B'=B_{0}.$ Therefore Irr$(G/N)$ belongs to $B_{0}.$
\end{proof}
Noation: For a finite group $G$, let $\pi (G)$ denote the set of all prime number dividing $|G|$.
By Lemma above, we have
\begin{cor}With the same notation as above,
Let $N$ be a normal subgroup of $G$. If $p-$principal block $(b_{0}(p),B_{0}(p))$ is a pair of Frobenius corresponding blocks for all $p \in \pi(N)$, then
$${\rm Irr}(G/N)\subseteq \bigcap_{p\in \pi(N)}{\rm Irr}B_{0}(p).$$
\end{cor}
Now we give a relation between $k(B)$ and $k(b)$, which denote the number of irreducible characters of $B$ and $b$, respectively.

\begin{thm}
Suppose $(b,B)$ is a pair of Frobenius corresponding blocks. Then we have
\begin{enumerate}
  \item \begin{equation*}
  k(B_{0})=k(G/N)+\frac{k(b_{0})-1}{|G/N|}
\end{equation*}
for principal block.
  \item
  \begin{equation*}
    k(B)=\frac{|N|k(b)}{|T(b)|}
\end{equation*}
for non-principal blocks.
\end{enumerate}
\end{thm}
\begin{proof}

1. Note that $f_{b_{0}}=e_{b_{0}}, \Omega_{f_{b_{0}}}=\Omega_{b_{0}}$ and
$$|{\rm Irr}(f_{b_{0}})|=|{\rm Irr}(b_{0})|=k(b_{0})$$
Take any $\chi \in {\rm Irr}(B_{0})$, which belongs to Irr$(G/N)$ or is induced from Irr$(b_{0})-\{1\}$. The set Irr$(b_{0})-\{1\}$ is divided into orbits under $G$-conjugate action. Each orbit has a length $|G/N|$. Hence we have our assertion.

2. First we know $B$ is the unique block to cover $b$ by Corollary 4.3. As we proved in the Lemma 5.1, any $\chi \in {\rm Irr}(G/N)$ belongs to a block $B'$, which covers the same principal block $b_{0}$ of $N$ as principal block $B_{0}$ of $G$. Hence Irr$(B)\bigcap$Irr$(G/N)=\emptyset.$ This tells us
all irreducible characters of Irr$(B)$ are induced from Irr$(\Omega_{f_{b}})$, which has $|G/T(b)|k(b)$ elements and divided into $k(B)$ orbits of $|G/N|$ elements.
\end{proof}

\begin{description}
  \item[Defect group] For a block $B\in {\rm Bl}(G)$, we use $D(B)$ to denote the defect group of $B$.
\end{description}

For the non-principal block $b\in {\rm Bl}(N)$, if $(b,B)$ is a pair of Frobenius corresponding blocks, we have
\begin{thm}
For a pair of Frobenius corresponding blocks $(b,B)$, we have
\begin{enumerate}
   \item For a non-principal block, $D(B)\leq N$.
   \item $D(b)=_{H}D(B), |T(b):N|\not\equiv 0 \; {\rm mod} p$
  \item  For a pair of principal Frobenius corresponding block $(b,B)$, we have that Sylow $p-$subgroup $P\leq N$ and $G/N$ is a $p'$-group.
\end{enumerate}
\end{thm}
\begin{proof}

1. In Theorem 5.3, we have proved the fact: every irreducible characters of $B$ is induced from $b$ for non-principal block. Hence every irreducible module $V$ affording irreducible character $\chi$ in $B$ is $N$-projective. By Green Theorem, we have $D(B)\leq N$.

2. According to results of Fong\cite[P345, Lemma 5.16]{naga} and part 1, the part 2 is easy to know.

3. Now we consider the principal block $(b,B)$. Since $1\neq\chi\in {\rm Irr}(b)$ is induced irreducibly to $\chi^{G}\in {\rm Irr}(B)$, we know that $b^{G}=B$ is Brauer correspondence from $b$ to $B$. Hence $B$ is the regular cover of $b$.

By the block participation of Cl$(N)$, we can find  that $\{e_{b}\widehat{C_{0}}|C_{0} \in \Omega_{b}\}$ is a basis of $e_{b}Z(FN)$ over $F$.
Since $T(b)=G$ and $e_{b}=e_{B}$, we claim that
$$\{e_{B}\widehat{C}|C_{0} \in \Omega_{b}\bigcap C,C\in {\rm Cl}(G)\}$$
 is the basis of $e_{B}(FN)^{G}$. First we have
$$e_{B}(FN)^{G}\leq e_{B}(FN)^{N}=e_{B}Z(FN)=e_{b}Z(FN)$$
which implies $e_{B}(FN)^{G}$ is generated by $\{e_{B}\widehat{C}|C_{0} \in \Omega_{b}\bigcap C,C\in {\rm Cl}(G)\}$ over $F$.   Second, $\{e_{B}\widehat{C}|C_{0} \in \Omega_{b}\bigcap C,C\in {\rm Cl}(G)\}$ is linearly independent  over $F$ by the fact that $\{e_{b}\widehat{C_{0}}|C_{0} \in \Omega_{b}\}$ is a basis of $e_{b}Z(FN)$ over $F$. Our claim is proved.

Suppose we have a block participation of Cl$(G)$. Let $\{e_{B}\widehat{C}|,C\in \Omega_{B}\}$ is the basis of $e_{B}Z(FG).$ Then we have a direct sum:
$$e_{B}Z(FG)=e_{B}(FN)^{G}\bigoplus e_{B}(F(G-N))^{G} $$
The summands in the above equality have  a basis  $\{e_{B}\widehat{C}|C\in \Omega_{B}\bigcap N\}$ and  a basis $\{e_{B}\widehat{C}|C\in \Omega_{B}\bigcap (G-N)\}$, respectively.
Therefore we have a basis
$$  \Delta= \{e_{B}\widehat{C}|C_{0} \in \Omega_{b}\bigcap C,C\in {\rm Cl}(G)\} \bigcup \{e_{B}\widehat{C}|C\in \Omega_{B}\bigcap (G-N)\} $$ for $e_{B}Z(FG)$ by our previous claim.

According to \cite{ols}, the number of classes $C$ in $\Delta$ with $D(C)=D(B)$ is equal to the number of major subsections of $B$. Then by \cite{bra}\cite{broue2}, the number of major subsections of $B$ is the number of the inertial group $T(\overline{B})$-conjugate classes in $Z(D)$, where $\overline{B}$ is the root of $B$ in $DC_{G}(D),D=D(B)$. From $D>1,$ it follows that the number is bigger than $1.$ Hence we can find at least one $1\neq C$ in $\Delta$, such that $D(C)=D(B).$

Now compute the linear $F-$representation $\omega^{*}_{B}:Z(FG)\rightarrow F(\widehat{C}\rightarrow \omega(\widehat{C})^{*})$. Since $B$ is the principal block,
we have $$\omega(\widehat{C})^{*}=(\frac{|G|}{|C_{G}(x)|})^{*},x\in C$$
where $*:R\rightarrow F=R/(\pi),p\in (\pi).$ Hence we have the following:
 $$\omega(\widehat{C})^{*}=0 \Leftrightarrow D(C)<|G|_{p}=|D(B)|.$$

Take $C\in \Delta.$ If $C\in G-N$, then $\omega(\widehat{C})^{*}=0$ by the fact that $B$ is the regular cover of $b$. Therefore we can find that
$C\in \Delta, 1\neq C_{0}\in \Omega_{b}\bigcap C$ such that $\omega(\widehat{C})^{*}\neq 0.$
Hence there exist $x\in C_{0}$ such that $D(B)\leq C_{G}(x)\leq N$ by Theorem 4.6.

\end{proof}

\begin{thm}Suppose $N$ is a normal subgroup of $G$. For each prime number $p||N|$, if the $p-$ principal block $(b_{0},B_{0})$ is a pair of Frobenius corresponding block, then we have $G=NH, N\bigcap H=1.$
\end{thm}
\begin{proof}
From Theorem 5.4, $G/N$ is a $p'-$group for each $p||N|$. Hence $N$ is a normal Hall-subgroup of $G$. By the Schur-Zassenhaus theorem, we conclude that there exists $H\leq G$ such that $G=NH,N\bigcap H=1.$
\end{proof}
\begin{exm}
Take $G=A_{4}$, the subgroup of symmetric group $S_{4}$ on the set with $4$ elements. It has $12$ elements and has a normal subgroup $$K_{4}=\{(1),(12)(34),(13)(24),(14)(23)\}.$$ Suppose $K_{4}$ has linear irreducible characters: $1,\lambda_{1},\lambda_{2},\lambda_{3}.$ It is easy to check that $A_{4}$ has three linear irreducible characters from $A_{4}/K_{4}: 1, \delta,\delta^{2}$. Another is $\chi=\lambda_{1}^{G}=\lambda_{2}^{G}=\lambda_{3}^{G}$, induced from $K_{4}.$ Hence $A_{4}$ is a Frobenius group with $K_{4}$ as kernel.

When $p=2,$ both $A_{4}$ and $K_{4}$ have only one block, principal block: $B_{0},b_{0}.$ Then we have
\begin{equation*}
  k(B_{0})=k(G/N)+\frac{k(b_{0})-1}{|G/N|}=3+1=4
\end{equation*}

When $p=3$, the kernel $N$ has $4$ blocks:
$${ \irr }(b_{0})=\{1\},{\irr}(b_{i})=\{\lambda_{i}\},T(b_{i})=N,i=1,2,3.$$
Hence $A_{4}$ has $2$ blocks:
\[{\irr} (B_{0})=\{1,\delta,\delta^{2}\}, {\irr} (B_{1})=\{\chi=\lambda_{1}^{G} \}\]
satisfying
\[   k(B_{0})=k(G/N)=3, k(B_{1})=\frac{|N|k(b_{1})}{|T(b_{1})|}=1\]
\end{exm}
\begin{exm}Take $G=S_{4}$, the symmetric group over $4$ elements. With the same notation as Example 5.6, $A_{4}$ has one irreducible character $\chi$ of dimension $3$ and three linear characters:$1,\delta,\delta^{2}.$ Since $A_{4}$ is the commutator subgroup of $S_{4}$, there are two linear characters: $1, \rho$ in Irr$(S_{4}/A_{4})\subseteq $ Irr$(S_{4})$. Consider inertial group $T(\chi)$ in $G$, it must be equal to $S_{4}$, because of $|S_{4}/A_{4}|=2.$ If a character $ \varphi$ in $G$ covers $\chi$, it must be a extension of $\chi$ to $G$. Hence we have two extensions: $\varphi,\varphi\rho,$ such that $\varphi_{A_{4}}=\chi,(\varphi\rho)_{A_{4}}=\chi$. Considering the inertial group $T(\delta)$ in $G,$ it must be $A_{4},$ as the linear character $\rho$ can not be its extension to $G$. Hence $\delta^{G}=(\delta^{2})^{G}=\zeta$ is an irreducible character of $S_{4}$. Now we have
$${\rm Irr}(S_{4})=\{1,\rho,\zeta,\varphi,\varphi\rho,\}$$with dimension $\{1,1,2,3,3\}.$

When $p=2$, the block in $H=A_{4}$ is  covered by a unique block in $G=S_{4},$ as $G/H$ is a $2-$group. Since $H$ has only one $2$-block, the principal block $b_{0}$, the principal $b_{0}$ is only covered by the principal block $B_{0}$ in $G$. Hence $G$ has only one $2$-block, the principal block $B_{0}.$ By arguments above, the character $\chi$ in Irr$(b_{0})$ is not induced irreducibly to $B_{0}$. Hence $(b_{0},B_{0})$ is not the Frobenius corresponding block.

When $p=3, H=A_{4}$ has two blocks:  \begin{itemize}
\item $b_{0}$ with irreducible characters $1, \delta,\delta^{2}$;
 \item $b_{1}$ with irreducible characters $\chi$.
 \end{itemize}
Since $\delta^{G}=\zeta$ is irreducible in $G$, $B_{0}$ is the unique block to cover $b_{0}$. The irreducible characters in $G/H$ must belong to blocks to cover principal block $b_{0}$, because  their restriction to $H$ only has the  trivial character $1$ as their irreducible constituents \cite[Page 340]{naga}. So $\rho \in $ Irr$(B_{0})$.  Both $\varphi$ and $\varphi\rho$ are irreducible characters of dimension $3$, hence each of them constitutes a block of defect zero in $G$, denoted by $B_{1}$ and $B_{2}$, respectively. Since $\chi$ is only covered by its extension $\varphi,\varphi\rho$ in $G$, both block $B_{1}$ and $B_{2}$ are the cover block of $b_{1}$.

Therefore, $(b_{0},B_{0})$ is Frobenius corresponding block. Neither $(b_{1},B_{1})$ nor $(b_{1}, B_{2})$  is  Frobenius corresponding block.
We also have
$$k(B_{0})=k(G/H)+\frac{k(b_{0})-1}{|G/H|}=3$$
\end{exm}

\section{Frobenius corresponding block and nilpotent property}

Since Frobenius group has a nilpotent kernel, we naturally need to study the connection between Frobenius corresponding blocks and nilpotent properties. According article \cite{bz1}\cite{bz2}, we will obtain some results in this section.

Notation: $\pi(G)$ denote the set of prime number dividing $|G|$. We state the following result
\begin{prop}For $p,q\in \pi(N)$, suppose that both $(b_{0}(p),B_{0}(p))$ and $(b_{0}(q),B_{0}(q))$ are pairs of Frobenius corresponding blocks. Then
\begin{align*}
{\rm Irr}(G/N)={\rm Irr}(B_{0}(p))\bigcap {\rm Irr}(B_{0}(q))\\
\Leftrightarrow \{1\}={\rm Irr}(b_{0}(p))\bigcap {\rm Irr}(b_{0}(q))
\end{align*}
\end{prop}
\begin{proof}
First by Lemma 5.1, we know that
$$ {\rm Irr}(G/N)\subseteq {\rm Irr}(B_{0}(p))\bigcap {\rm Irr}(B_{0}(q)).$$

If there exists $1\neq \varphi \in {\rm Irr}(b_{0}(p))\bigcap {\rm Irr}(b_{0}(q))$, then we have that
$$\varphi^{G}\in {\rm Irr}(B_{0}(p))\bigcap {\rm Irr}(B_{0}(q))$$
and $\varphi^{G} \not\in {\rm Irr}(G/N).$

Conversely, if there exists a $\chi \in {\rm Irr}(B_{0}(p))\bigcap {\rm Irr}(B_{0}(q))$ and $\chi \not\in {\rm Irr}(G/N), $ then we have that $1\neq\varphi \in {\rm Irr}(b_{0}(p))$ such that $\varphi^{G}=\chi$ and $1\neq\psi \in {\rm Irr}(b_{0}(q))$ such that $\psi^{G}=\chi.$ Hence we have $\psi^{g}=\varphi, g\in G$ by Clifford Theorem. It means that
$\varphi \in {\rm Irr}(b_{0}(q))^{g}$. but ${\rm Irr}(b_{0}(q))^{g}={\rm Irr}(b_{0}(q)).$
Hence we get
$$1\neq \varphi \in {\rm Irr}(b_{0}(p))\bigcap {\rm Irr}(b_{0}(q)).$$
\end{proof}

In \cite{bz1}\cite{bz2}, authors define
$$B_{0}(G)_{\pi}=\bigcap_{p \in \pi}\textrm{Irr}(B_{0}(p)),\pi\subseteq \pi(G)$$
and denote $B_{0}(G)$ for $\pi=\pi(G).$ If $B_{0}(G)_{\pi}=\{1_{G}\}$, they say that $\textrm{Irr}(G)$ is principally $\pi$-separated. If $B_{0}(G)=\{1_{G}\}$, it is said that $\textrm{Irr}(G)$ is principally separated. According to result in \cite{bz1}\cite{bz2}, we have the following result by Proposition 6.1 above:

\begin{thm}Suppose that $(b_{0}(p),B_{0}(p))$ is a pair of Frobenius corresponding blocks for any $p\in \pi(N)$. Then $N$ is nilpotent if and only if
$$ {\rm Irr}(G/N)= {\rm Irr}(B_{0}(p))\bigcap {\rm Irr}(B_{0}(q)),p,q\in \pi(N)$$

\end{thm}
\begin{proof}
This is because $N$ is nilpotent if and only if ${\rm Irr}(N)$ is principally $\{p,q\}$-separated for any two different prime numbers $p,q \in \pi(N).$ Then by Proposition 6.1, we have the result.
\end{proof}
Theorem 6.2 induces the following result easily:
\begin{cor}With the same conditions as in Theorem 6.2, If $N$ is nilpotent, then
$${\rm Irr}(G/N)=\bigcap_{p\in \pi(N)}{\rm Irr}(B_{0}(p))$$
\end{cor}
\begin{proof}
This is because
$$\textrm{Irr}(G/N)\subseteq \bigcap_{p\in \pi(N)}{\rm Irr}(B_{0}(p))\subseteq \textrm{Irr}(B_{0}(p))\bigcap \textrm{Irr}(B_{0}(q))$$
for any $p,q \in \pi(N).$
\end{proof}

 If  $\textrm{Irr}(G)$ is not principal separated, one of  questions in\cite{bz2} is: How big can $B_{0}(G)$ be?

The following result is related to their question. According to results in  \cite[X, Theorem 1.5, P416]{feit} and \cite{bz2}, we are going to prove the following result:
\begin{thm}
Let $N\unlhd G$ and $G/N$ is solvable. Suppose $(b_{0}(p),B_{0}(p))$ is a pair of Frobenius corresponding blocks for all $p \in \pi(N)$. Let $H=\prod_{p\in \pi(G/N)}O_{p'}(G)$. Then
$${\rm Irr}(G/H)=B_{0}(G).$$
In particular, ${\rm Irr}(G)$ is principally separated if and only if $G=\prod_{p\in \pi(G/N)}O_{p'}(G).$
\end{thm}
\begin{proof}
First we claim a general case for any $\pi\leq \pi(G)$:
$$\textrm{Irr}(G/H)=\bigcap_{p\in \pi}\textrm{Irr}(G/O_{p'}(G)).$$
Since $O_{p'}(G)\leq H,$ we have
$$\textrm{Irr}(G/H) \subseteq \bigcap_{p\in \pi}\textrm{Irr}(G/O_{p'}(G)).$$
If $\chi \in \bigcap_{p\in \pi}\textrm{Irr}(G/O_{p'}(G))$, it means that: $H\leq \textrm{ker}\chi$ and
$$\bigcap_{p\in \pi}\textrm{Irr}(G/O_{p'}(G))\subseteq \textrm{Irr}(G/H).$$ Hence our claim is proved.

If $G$ is $p-$solvable, we have
$\textrm{Irr}(G/O_{p'}(G))=\textrm{Irr}(B_{0}(p))$.

Now take $\pi=\pi(G/N)$. By arguments above and the given conditions, we have
$${\rm Irr}(G/H)=\bigcap_{p\in \pi(G/N)}\textrm{Irr}(B_{0}(p)).$$

Since $(b_{0}(p),B_{0}(p))$ is a pair of Frobenius corresponding blocks for all $p\in \pi(N)$, we know that $N\unlhd H$ by  Theorem 5.4. Hence $\textrm{Irr}(G/H)\subseteq \textrm{Irr}(G/N)$.

By corollary 5.2, we have
 $$\textrm{Irr}(G/N)\subseteq \textrm{Irr}(B_{0}(p)), p\in \pi(N).$$

Note
$$\pi(G)=\pi(N)\bigcup \pi (G/N),\pi(N)\bigcap \pi (G/N)=\emptyset$$ by Theorem 5.4.
 Then we have
\begin{align*}
B_{0}(G)&=\bigcap_{p\in \pi(G)}\textrm{Irr}(B_{0}(p))\\
& = \bigcap_{p\in \pi(G/N)}\textrm{Irr}(B_{0}(p)) \bigcap_{p\in \pi(N)}\textrm{Irr}(B_{0}(p))\\
& = \textrm{Irr}(G/H)\bigcap_{p\in \pi(N)}\textrm{Irr}(B_{0}(p))\\
& =  \textrm{Irr}(G/H)
\end{align*}
Our conclusion is proved.
\end{proof}

\bibliographystyle{amsplain}
\bibliography{xbib}
\bibliographystyle{amsplain}

\end{document}